\documentclass[11pt, a4paper, reqno]{amsart}

\usepackage[utf8]{inputenc}
\usepackage[english]{babel}
\usepackage{url}
\usepackage{xcolor}
\usepackage[all,cmtip]{xy}
\usepackage{graphicx}

\usepackage{hyperref}

\usepackage{amsmath,amssymb,enumerate,mathrsfs,mathtools,amsthm,emptypage,array,xpatch,faktor,xfrac,multirow,tikz}
\usepackage[normalem]{ulem}
\usepackage{pifont}

\calclayout

\newcommand{\vol}{\operatorname{vol}}
\newcommand{\NE}{\operatorname{NE}}

\newcommand{\Pic}{\operatorname{Pic}}

\newcommand{\codim}{\operatorname{codim}}

\newcommand{\Nef}{\operatorname{Nef}}
\newcommand{\ph}{\varphi}
\newcommand{\w}{\widetilde}
\newcommand{\ma}{\mathcal}

\newcommand{\ol}{\mathcal{O}}

\newcommand{\pr}{\mathbb{P}}

\newcommand{\R}{\mathbb{R}}
\newcommand{\Z}{\mathbb{Z}}
\newcommand{\N}{\mathcal{N}_1}

\newcommand{\Aut}{\operatorname{Aut}}

\newcommand{\xdasharrow}[2][->]{
\tikz[baseline=-\the\dimexpr\fontdimen22\textfont2\relax]{
\node[anchor=south,font=\scriptsize, inner ysep=1.5pt,outer xsep=2.2pt](x){#2};
\draw[shorten <=3.4pt,shorten >=3.4pt,dashed,#1](x.south west)--(x.south east);
}
}

\newtheorem{thm}{Theorem}[section]
\newtheorem{corollary}[thm]{Corollary}
\newtheorem{proposition}[thm]{Proposition}
\newtheorem{lemma}[thm]{Lemma}

\newtheoremstyle{remstyle}
  {}
  {}
  {}
  {}
  {\bfseries}
  {.}
  { }
  {}
\theoremstyle{remstyle}
\newtheorem{remark}[thm]{Remark}
\newtheorem{definition}[thm]{Definition}
\newtheorem{step}[thm]{Step}

\title{\small K-polystability of Fano 4-folds with large Lefschetz defect}
\author{Eleonora A.\ Romano}
\address{E.\ A.\ Romano:  Universit\`a di Genova, Dipartimento di Matematica, Dipartimento di eccellenza 2023--2027, via Dodecaneso 35, 16146 Genova - Italy}
\email{eleonoraanna.romano@unige.it}

\author{Saverio A.\ Secci}
\address{S.\ A.\ Secci: Universit\`a di Milano, Dipartimento di Matematica, via Saldini 50, 20133 Milano - Italy}
\curraddr{SISSA - Scuola Internazionale Superiore di Studi Avanzati, Via Bonomea 265, 34136 Trieste, Italy}
\email{ssecci@sissa.it}

\date{\today}
\subjclass[2020]{14J45, 14J35, 32Q20}

\begin{document}
\maketitle

\begin{abstract}
In this paper, we investigate K-polystability on smooth complex Fano 4-folds with Lefschetz defect at least 2, focusing on the case of Lefschetz defect 3 and on Casagrande-Druel Fano 4-folds with Lefschetz defect 2. We show that exactly 5 of the 19 families of Fano 4-folds with Lefschetz defect 3 are K-polystable. Moreover, among 175 families of Casagrande-Druel Fano 4-folds with Lefschetz defect 2, we prove that 5 are K-polystable, while 132 are K-unstable.
\end{abstract}

\section{Introduction}

The notion of K-stability was first introduced in \cite{Tian97} as a criterion to characterize the existence of a K\"ahler-Einstein metric on complex Fano varieties, and has been later formulated using purely algebro-geometric terms in \cite{Don02}. Nowadays, by the celebrated works \cite{Chen2015, Tian2015}, it is well known that a complex smooth Fano variety admits a K\"ahler-Einstein metric if and only if it is K-polystable. 

This correspondence links together differential and complex algebraic geometry, and it represents one of the main motivations to investigate K-polystability of Fano varieties. Moreover, the condition of K-stability has been successfully used to construct moduli spaces of Fano varieties, thus increasing its relevance within modern algebraic geometry (see \cite[$\S$2]{Xu2021} and references therein). We refer to \cite{Xu2021} for the original definitions of K-stability involving $\mathbb{C}^*$-degenerations of Fano varieties, and for a survey on this topic from an algebro-geometric viewpoint. More recently, in \cite{Boucksom17} valuation methods have been introduced to reinterpret one parameter group degenerations: these  new techniques gave a fundamental development to the algebraic theory of K-stability, due to equivalent and easier ways to test K-stability notions in many situations, such as the computation of the $\beta$-invariant of divisors over the target variety (see \cite{Fujita19, Li17}). Indeed, the $\beta$-invariant (see Definition \ref{def:beta}) may be explicitly computed for many classes of Fano varieties whose structure of divisors in their birational models is well understood. 

The situation is completely known for del Pezzo surfaces (see Corollary \ref{cor_delPezzo}), while we refer to \cite{Araujo23} for the case of Fano 3-folds and for a general and updated literature on this topic.  

In this paper, we use valuation methods related to the $\beta$-invariant to study K-polystability of Fano 4-folds with Lefschetz defect $\delta\geq 3$ and some families of Fano 4-folds with $\delta=2$. We refer to \cite{Casa12} for the definition of the Lefschetz defect of a Fano manifold (Definition \ref{def:Lefschetz}) and for the classification when $\delta\geq 4$.

The classification of smooth Fano 4-folds with $\delta=3$ has been partially carried out in \cite{CR} and then completed in \cite[Proposition 1.5]{CRS}. Instead there is no complete classification of smooth Fano 4-folds with $\delta=2$ yet. By \cite[Theorem 1.8]{CS24} we know that smooth Fano 4-folds with $\delta=2$ have Picard number $\rho\leq 6$ and this bound is sharp. In this paper, among smooth Fano 4-folds with $\delta=2$, we are interested in those obtained through the construction introduced by Casagrande and Druel \cite{CD}. These Fano 4-folds have been classified in \cite{Secci23} when they have Picard number $\rho=3$, and by \cite{Pier25} when $\rho\in\{4,5,6\}$: they form $175$ families.

From the viewpoint of K-polystability, the case of Fano 4-folds with $\delta\geq 4$ easily follows from known results. Indeed, by \cite[Theorem 3.3]{Casa12} these varieties are products of two del Pezzo surfaces, and applying \cite[Theorem 1.1]{Zhuang20} (Lemma \ref{lem:product} below) we see that a product of Fano varieties is K-polystable if and only if both of its factors are (see Remark \ref{rem:delta4} for details).

Thus, it arises our motivation to study the subsequent case of Fano 4-folds having Lefschetz defect $\delta=3$: among the possible 19 families of such Fano 4-folds classified in \cite{CR} and \cite{CRS}, we establish which ones are K-polystable. We state our conclusions in the following result. Note that $Y$ from item (ii) below is $\ma N^{\circ} 5.1$ in the notation of \cite{Araujo23}.
\begin{thm} \label{thm:main} Let $X$ be a smooth Fano 4-fold with $\delta_X \geq 3$. Denote by $F'$ (resp.\ $F$) the blow-up of $\mathbb{P}^2$ along two (resp.\ three non-collinear) points. Then,
\begin{itemize}
\item [(i)] if $\delta_X \geq 4$, then $X$ is K-polystable if and only if $X \ncong S \times  \mathbb{F}_1$, $X\ncong S \times F'$, with $S$ a del Pezzo surface having $\rho_S=\delta_X+1$;
\item [(ii)] if $\delta_X=3$, then $X$ is K-polystable if and only if it is one of the following: 
\begin{itemize}
\item[{\tiny \textbullet}] $X\cong \mathbb{P}^2\times F$;
\item[{\tiny \textbullet}] $X\cong \mathbb{P}^1\times\mathbb{P}^1 \times F$;
\item[{\tiny \textbullet}] $X\cong F\times F$;
\item[{\tiny \textbullet}] $X$, the blow-up of $\mathbb{P}^1\times \mathbb{P}_{\mathbb{P}^1\times\mathbb{P}^1}(\mathcal{O}\oplus \mathcal{O}(1,-1))$ along two surfaces isomorphic to $\mathbb{P}^1\times \mathbb{P}^1$;
\item[{\tiny \textbullet}] $X\cong \mathbb{P}^1\times Y$, where $Y$ is the blow-up of $\mathbb{P}^3$ along a disjoint union of a line and a conic, and along two non-trivial fibers of the exceptional divisor over the blown-up line. 
\end{itemize}
\end{itemize}
\end{thm}

Finally, we study K-polystability of Casagrande-Druel Fano 4-folds with Lefschetz defect $\delta=2$ and show the following result. 
\begin{thm}\label{thm:main2}
Let $X$ be a smooth Casagrande-Druel Fano 4-fold with $\delta_X=2$. Then there exists a prime divisor $G\subset X$ such that $\beta(G) \leq 0$, and equality holds if and only if $X$ is one of the following:
\begin{itemize}
\item [(i)] $X^1_{1,2}$, $X^2_{1,2}$, $X^3_{1,2}$, $X^4_{1,2}$, $X^5_{1,2}$, $X^6_{1,2}$, $X^6_{2,4}$, $X^7_{1,2}$, $X^7_{2,4}$, $X^7_{3,6}$ from \cite{Secci23} or
\item [(ii)] $\#4$-$2$, $\#5$-$2$, $\#6$-$4$, $\#7$-$10$, $\#7$-$17$, $\#7$-$21$, $\#7$-$23$, $\#8$-$16$, $\#8$-$18$, $\#9$-$4$, $\#10$-$5$, $\#11$-$2$, $\#11$-$3$, $\#12$-$2$, $\#13$-$2$, $\#14$-$4$, $\#15$-$3$, $\#16$-$4$, $\#17$-$3$, $\#18$-$4$, $\#19$-$5$, $\#19$-$7$, $\#19$-$9$, $\#20$-$8$, $\#20$-$10$, $\#21$-$14$, $\#21$-$19$, $\#21$-$21$, $\#21$-$35$, $\#21$-$37$, $\#22$-$15$, $\#22$-$17$, $\#23$-$5$ from \cite{Pier25}.
\end{itemize}
In particular, $X$ is K-unstable if is not in (i) or (ii).
Moreover, $\#7$-$10$,  $\#7$-$17$, $\#8$-$16$, $\#21$-$14$ and $\#21$-$21$ are a product $\pr^1\times$(Fano 3-fold) and are K-polystable.
\end{thm}

\noindent \textbf{Outline of the paper.} After giving some preliminaries on K-polystability, on the Lefschetz defect $\delta$ and on the structure of Fano 4-folds having $\delta=3$ in Sections \ref{sec:prelim} and \ref{sec:delta3}, we dedicate Section \ref{sec:proof} to the proof of Theorem \ref{thm:main}: as we have already observed, proving item (ii) will require the most effort. Finally, Sections \ref{sec:delta2} and \ref{sec:proof2} contain the classification of Casagrande-Druel Fano 4-folds and the proof of Theorem \ref{thm:main2}.

\noindent \textbf{Strategy of proof.} To prove Theorem \ref{thm:main} we distinguish between the toric and the non-toric case, proceeding in two different ways. The key point to study the toric case is a well known criterion on K-polystability for toric Fano varieties (see Lemma \ref{lem:toric}). Note that in Theorem \ref{thm:main}(ii) all but the last Fano 4-fold are toric varieties.
The non-toric case, on the other hand, consists of five families and is more challenging to check. Here we use the Fujita-Li's valuative criterion (see Theorem \ref{thm_beta}). Our strategy is to prove Proposition \ref{prop:beta}, which provides an explicit formula for the $\beta$-invariant on a special exceptional divisor, denoted by $\widetilde{D}$, contained in all non-toric Fano 4-folds with $\delta=3$. We introduce and describe $\widetilde{D}$, as well as the geometry of its ambient variety, in $\S$\ref{sec:relB}.
To deduce the formula in Proposition \ref{prop:beta}, we first determine the Zariski decomposition of $-K_X-t\widetilde{D}$ for $t\geq 0$ (Proposition \ref{prop:ZD}) through several technical lemmas, in which we heavily rely on our knowledge of the birational geometry of Fano 4-folds with $\delta=3$.
Finally, we deduce that four out of five families of non-toric Fano 4-folds with $\delta=3$ are not K-polystable, as the $\beta$-invariant on $\widetilde{D}$ turns out to be negative. The remaining case (that is the fifth variety in our list (ii) from Theorem \ref{thm:main}) is isomorphic to a product and it gives the only example of non-toric K-polystable Fano variety with $\delta=3$: we apply Lemma \ref{lem:product} to deduce its K-polystability. We summerize our conclusions in Section \ref{sec:tables}, see Table \ref {table:delta3} and Table \ref{table:high-delta}.

The proof of Theorem \ref{thm:main2} follows the same strategy: we produce an explicit formula for the $\beta$-invariant on a pair of special exceptional divisors, denoted by $G$ and $\hat G$, via the Zariski decomposition (Proposition \ref{prop:ZD-G}) and show that $\beta(\hat G)=-\beta(G)$ (Proposition \ref{prop:beta-G}). We determine the cases when $\beta(G)=\beta(\hat{G})=0$ and note from \cite{Pier25} that five of these families are products of $\mathbb{P}^1$ with a smooth Fano 3-fold; Lemma \ref{lem:product} then implies their K-polystability. The remaining families with vanishing $\beta$-invariant at $G$ and $\hat G$ are non-toric, so Lemma \ref{lem:toric} does not apply.

\medskip

\noindent \textbf{Notations.} We work over the field of complex numbers. Let $X$ be a smooth projective variety. 

{\tiny \textbullet} $\sim$ denotes linear equivalence for divisors. We will often not distinguish between a Cartier divisor $D$ and its corresponding invertible sheaf $\ol_X(D)$.

{\tiny \textbullet} $\mathcal{N}_{1}(X)$ (resp.\ $\mathcal{N}^{1}(X)$) is the $\mathbb{R}$-vector space of one-cycles (resp.\ divisors) with real coefficients, modulo numerical equivalence, and
 $\rho_{X}:=\dim \mathcal{N}_{1}(X)=\dim \mathcal{N}^{1}(X)$ is the Picard number of $X$. Sometimes we denote it simply by $\rho$.

{\tiny \textbullet} The \textit{pseudoeffective cone} is the closure of the cone in $\mathcal{N}^{1}(X)$ generated by the classes of effective divisors on $X$; its interior is the big cone. An $\mathbb{R}$-divisor is called pseudoeffective if its numerical class belongs to the pseudoeffetive cone.

{\tiny \textbullet} We denote by $[C]$ the numerical equivalence class in $\mathcal{N}_{1}(X)$ of a one-cycle $C$ of $X$. 

{\tiny \textbullet} $\operatorname{NE}(X)\subset \mathcal{N}_{1}(X)$ is the convex cone generated by classes of effective curves. 

{\tiny \textbullet } A \textit{contraction} of $X$ is a surjective morphism $\varphi\colon X\to Y$ with connected fibers, where $Y$ is normal and projective.
 
{\tiny \textbullet} The \textit{relative cone} $\text{NE}(\varphi)$ of $\varphi$ is the convex subcone of $\text{NE}(X)$ generated by classes of curves contracted by $\varphi$.

{\tiny \textbullet} We denote by $\delta_X$, or simply by $\delta$, the $\text{Lefschetz defect}$ of $X$. 

\medskip

\noindent\textbf{Acknowledgements.} E.A.R.\ is supported by the MIUR Excellence Department Project awarded to Dipartimento di Matematica, Università di Genova, CUP D33C23001110001. She dedicates this work to her  daughter, Miriam.

S.A.S.\ is grateful to the Università di Genova for the kind hospitality and support provided during part of the preparation of this work. Most of the work was carried out during his position at the Università di Milano, and the second version of the article was written while he was in SISSA.

Both authors are members of GNSAGA, INdAM.

\section{Preliminaries}\label{sec:prelim}
This section collects the preliminaries on K-polystability in $\S$\ref{sub_K_stab}, and on the Zariski decomposition in Mori dream spaces in $\S$\ref{sec:ZD}.

\subsection{Fujita-Li's valuative criterion} \label{sub_K_stab}
In this subsection, we recall the characterization of K-semistability via valuations and gather some preliminary results arising from this perspective. 
A key definition is the invariant $\beta(E)$, the $\beta$-invariant, computed on a divisor $E$ over $X$, that is a divisor on a normal birational model $Y$ over $X$ (see \cite{Fujita19, Li17}). Here we focus on smooth varieties, although the treatment can be extended to $\mathbb{Q}$-Fano varieties.

\begin{definition}\label{def:beta} Let $X$ be a smooth Fano variety and $E$ a prime divisor on a normal birational model $\mu\colon Y\to X$. We define
$$\beta(E):=A(E)- \frac{1}{{(-K_X)^n}} \int_{0}^{\infty} \vol(-\mu^*K_X-tE) dt$$
where $A(E)$ is the log-discrepancy of $X$ along $E$, namely $A(E):=1+\text{ord}_E(K_Y-\mu^*(K_X))$.
\end{definition}
We refer to \cite[$\S$2.2.C]{LazI} for the definition of $\vol(\, \text -\, )$. For simplicity, we set $$S(E):=\frac{1}{{(-K_X)^n}} \int_{0}^{\infty} \vol(-\mu^*K_X-tE) dt$$ and notice that this integral takes values in a closed set $[0, \tau]$, where $\tau=\tau(E)$ is the pseudoeffective threshold of $E$ with respect to $-K_X$, namely
\begin{equation*}
\tau(E):=\text{sup}\{s\in \mathbb{Q}_{>0}| -\mu^*K_X-sE \  \text{is big}\}.
\end{equation*}
Therefore, we have $\beta(E)=A(E)-S(E)$. 

The importance of the $\beta$-invariant mainly arises from the following result that is known as the valuative criterion for K-(semi)stability, and it is due to Fujita and Li \cite{Fujita19, Li17}, to which we also refer for a more general statement.  

\begin{thm} \label{thm_beta} Let $X$ be a smooth Fano variety. Then $X$ is K-semistable if and only if $\beta(E)\geq 0$ for all divisors $E$ over $X$. 
\end{thm}
For our purposes, we will use that if $X$ is not K-semistable, then it is not K-polystable by definition.
Using the valuative criterion, it is easy to deduce that among del Pezzo surfaces, $\mathbb{F}_1$ and the blow-up of $\mathbb{P}^2$ at two points are not K-polystable (see for instance \cite[Lemma 2.3, Lemma 2.4]{Araujo23}). More precisely, we have the following:
\begin{corollary} \cite{Tian90} \label{cor_delPezzo} Let $S$ be a del Pezzo surface. Then $S$ is K-polystable if and only if $S$ is neither isomorphic to  $\mathbb{F}_1$ nor to the blow-up of $\mathbb{P}^2$ at two points.
\end{corollary}
Many varieties that we are going to study are products, and so we recall  the following result. We refer to \cite[Theorem 1.1]{Zhuang20} for a more general statement involving the other notions of K-stability. 
\begin{lemma} \cite[Theorem 1.1]{Zhuang20} \label{lem:product} Let $X_1$, $X_2$ be Fano varieties and let $X=X_1\times X_2$. Then $X$ is K-polystable if and only if $X_i$ is K-polystable for $i=1,2$. 
\end{lemma}

\begin{remark} \label{rem:beta}
Although the computation of the $\beta$-invariant involves the volume of divisors that are not necessarily nef (in fact, we use the Zariski decomposition; see §\ref{sec:ZD}), it may still be possible to compute it explicitly for divisors whose behavior on birational models of the ambient variety is well understood, thanks to the powerful tools of birational geometry. This is the strategy we will employ in the proof of Theorem \ref{thm:main} (see Proposition \ref{prop:beta} and the proof of Proposition \ref{prop:non-toric}), as well as in the proof of Theorem \ref{thm:main2} (see Proposition \ref{prop:beta-G}).
\end{remark}

\subsection{Zariski decomposition in Mori dream spaces} \label{sec:ZD} A common approach to compute the $\beta$-invariant of an effective divisor on a Fano variety, and hence its volume, is to determine its Zariski decomposition. In our case, we note that smooth Fano varieties are Mori dream spaces (MDS) by \cite{BCHM}. In fact, the existence of such a well-behaved decomposition characterizes Mori dream spaces, and on such varieties the Zariski decomposition is unique, as observed in \cite[Remark 2.12]{OKA2016}. To make our exposition self-contained, we begin with the following basic definition, see \cite[$\S$2]{OKA2016} for details. 

\begin{definition} Let $X$ be a normal projective variety and $D$ a pseudoeffective $\mathbb{Q}$-Cartier $\mathbb{Q}$-divisor on $X$. A Zariski decomposition of $D$ is given by a pair of $\mathbb{Q}$-Cartier $\mathbb{Q}$-divisors $P$ and $N$ on $X$ which satisfy the following properties:
\begin{itemize}\renewcommand\labelitemi{\tiny \textbullet}
\item $P$ is nef;
\item $N$ is effective;
\item $D$ is $\mathbb{Q}$-linearly equivalent to $P+N$;
\item for any sufficiently divisible $m\in \mathbb{Z}_{>0}$ the multiplication map $$H^0(X, \mathcal{O}(mP))\to H^0(X, \mathcal{O}(mD))$$ given by the tautological section of $\mathcal{O}(mN)$ is an isomorphism.
\end{itemize}
\end{definition}
If $X$ is a MDS, by \cite[Proposition 1.11(2)]{HK00} we know that there exist finitely many birational contractions $g_i\colon X\dashrightarrow  X_i$ and that the pseudoeffective cone of $X$ is given by the union of finitely many Mori chambers $\mathcal{C}_i$. Each chamber is of the form
$$\mathcal{C}_i={g_i}^* \Nef{(X_i)+\mathbb{R}_{\geq 0}}\{E_1, \dots, E_k\}$$
with $E_1, \dots, E_k$ prime divisors contracted by $g_i$, and $\Nef{(X_i)}$ the nef cone of $X_i$.

We may now interpret such a result as an instance of Zariski decomposition, as done in \cite[Proposition 2.13]{OKA2016}. Indeed, for every pseudoeffective $\mathbb{Q}$-Cartier $\mathbb{Q}$-divisor $D$ on a MDS $X$, there exists a rational birational contraction $g\colon X\dashrightarrow Y$ (factorizing through an SQM and a birational contraction $X\xdasharrow{$\psi$} X^{\prime}\xrightarrow{g'} Y$) and $\mathbb{Q}$-Cartier $\mathbb{Q}$-divisors $P$ and $N$ on $X$ such that $D $ is $\mathbb{Q}$-linearly equivalent to $P+N$, $P':=\psi_*P$ is nef on $X^{\prime}$ and defines $g'\colon X'\to Y$, $N':=\psi_*N$ is $g'$-exceptional, and the multiplication map
$$H^0(X', mP')\to H^0(X', m(\psi_*D))$$
is an isomorphism for $m\gg0$. Namely, $P'$ and $N'$ give a Zariski decomposition of $\psi_*D$ as a divisor in $X^{\prime}$.  To see this, we simply set $P:=g^*g_*D$ and $N:=D-P$.

\section{Fano manifolds with Lefschetz defect 3}  \label{sec:delta3}
In this section we recap the classification and construction of smooth complex Fano varieties with Lefschetz defect $\delta=3$.

The Lefschetz defect $\delta_X$ of a smooth Fano variety $X$ is an invariant that depends on the Picard number of its prime divisors, and it was first introduced in \cite{Casa12}. We recall its definiton below, see also \cite{Casa2023} for a recent survey on this new invariant and its properties.
\begin{definition}\label{def:Lefschetz}
Let $X$ be a complex smooth Fano variety, and $D$ be a prime divisor on $X$. Consider the pushforward $\iota_*\colon\N(D)\to\N(X)$ induced by the inclusion and set $\N(D,X):=\iota_*(\N(D))$. The Lefschetz defect of $X$ is
$$\delta_X:=\max\bigl\{\codim\N(D,X)\,|\,D\text{ a prime divisor in }X\bigr\}.$$
\end{definition}

\begin{remark} \label{rem:delta4} Smooth Fano varieties with large Lefschetz defect have been completely described in arbitrary dimension: indeed, $X$ has a rigid geometry when $\delta_X \geq 4$, that is $X$ is the product of Fano varieties of lower dimension (cf. \cite[Theorem 3.3]{Casa12}).
In particular, if $X$ is a Fano 4-fold having $\delta_X \geq 4$, then $X\cong S_1\times S_2$ with $S_i$ del Pezzo surfaces, and applying \cite[Example 3.1]{Casa12} we may assume that $\rho_{S_1}=\delta_X+1$. Then, by Corollary \ref{cor_delPezzo} and Lemma \ref{lem:product} we conclude that $X$ is K-polystable if and only if $S_i$ is neither isomorphic to $\mathbb{F}_1$ nor to the blow-up of $\mathbb{P}^2$ at two points.
\end{remark}

Thus, we consider the next case, i.e.\ Fano $4$-folds with $\delta=3$. The strategy to prove Theorem \ref{thm:main} is to compute the $\beta$-invariant on a particular divisor that these varieties carry out. We see that this invariant turns out to be negative in many examples, so that we understand when $K$-polystability fails thanks to Theorem \ref{thm_beta}. Although not necessarely a product, Fano varieties with $\delta=3$ still have a very explicit description, indeed by \cite[Theorem 1.4]{CRS} they are obtained via two possible constructions that we are going to recall below (cf. \cite[$\S$3, $\S$4]{CRS}). 

Let $X$ be a smooth Fano variety with $\delta_X=3$. Then, there exist a smooth Fano variety $T$ with $\dim T=\dim X-2$ and a $\pr^2$-bundle $\ph\colon Z\to T$, such that $X$ is obtained by blowing-up $Z$ along three pairwise disjoint smooth, irreducible, codimension 2 subvarieties $S_1,S_2,S_3$; we will denote by $h\colon X\to Z$ the blow-up map and set $\sigma:=h \circ \varphi\colon X\to T$. The $\pr^2$-bundle $\ph\colon Z\to T$ is the projectivization of a suitable decomposable vector bundle on $T$, and $S_2$ and $S_3$ are sections of $\ph$. Instead, $\ph_{|S_1}\colon S_1\to T$ is finite of degree $1$ or $2$: this yields two distinct constructions depending on the degree of $S_1$ over $T$, whenever the degree is 1 we refer to it as Construction A, otherwise we get Construction B.

As a consequence, in \cite{CR} and \cite{CRS} we get the complete classification in the case of dimension $4$ and $\delta=3$, as follows. In Theorem \ref{thm:main} we are going to analyze the $K$-polystability for all of these families. 

\begin{thm}[\cite{CR}, Theorem 1.1; \cite{CRS}, Proposition 1.5]
\label{classification} Let $X$ be a Fano $4$-fold with $\delta_X=3$. Then $5\leq \rho_X \leq 8$ and there are $19$ families for $X$, among which $14$ are toric. 
\begin{itemize}\renewcommand\labelitemi{\tiny \textbullet}
\item If $\rho_X=8$, then $X\cong F\times F$,  where $F$ is the blow-up of $\mathbb{P}^2$ at $3$ non-collinear points;
\item if $\rho_X=7$, then $X\cong F'\times F$, where $F'$ is the blow-up of $\mathbb{P}^2$ at $2$ points;
\item if $\rho_X=6$, there are $11$ families for $X$, among which $8$ are toric;
\item if $\rho_X=5$, there are $6$ families for $X$, among which $4$ are toric.
\end{itemize}
\end{thm}

\begin{remark}\label{rem:toric} In view of \cite[Remark 6.1]{CRS}, the toric families of Theorem \ref{classification} are exaclty those arising via Construction A. More precisely, they correspond to the products $F\times F$ and $F'\times F$ if $\rho \geq 7$ and, following Batyrev's classification of smooth toric Fano 4-folds and its notation (see \cite{Baty}), to the toric varieties of type $U$ (eight possible families) if $\rho=6$, and to the toric varieties of type $K$ (four possible families) if $\rho=5$. For most of these cases we will use a characterization result on K-polystability for toric varieties (see $\S$\ref{sec:toric_proof}). Thus, the most effort will be required by the Fano 4-folds obtained via Construction B, that is the non-toric families. The two non-toric families with $\rho=5$ have been studied in \cite[Examples 5.1 and 5.2]{CR}, while the remaining three families with $\rho=6$ are described in \cite[$\S$7]{CRS}. 
\end{remark}

\subsection{Construction B: relative cone and relative contractions} \label{sec:relB} Construction B is described in \cite[$\S$4]{CRS}, we summarize it in the following.
We have
$$\varphi\colon Z\cong \pr_T(\mathcal{O}(N)\oplus\mathcal{O} \oplus\mathcal{O}) \to T,$$
where $N$ is a divisor on $T$ such that $h^0(T,2N)>0$ and $-K_T\pm N$ is ample. We denote by $H$ a tautological divisor of $Z$. Let $D:=\mathbb{P}(\mathcal{O} \oplus\mathcal{O})\hookrightarrow Z$ be the divisor given by the projection $\mathcal{O}(N)\oplus\mathcal{O} \oplus\mathcal{O} \to \mathcal{O} \oplus\mathcal{O}$, so that $D\cong \mathbb{P}^1\times T$ and $D\sim H -\varphi^*N$. Let now $S_2$, $S_3\subset D$, $S_i\cong \{pt\}\times T\subset D$, be the sections corresponding to the projections $\mathcal{O} \oplus\mathcal{O}\to \mathcal{O}$, while $\varphi_{\mid S_1}\colon S_1\to T$ is a double cover ramified along $\Delta \in |2N|$ (see \cite[Remarks 4.1, 4.3]{CRS}). There exists a unique smooth divisor $H_0\in |H|$ containing $S_1$ such that $H_0\cong \mathbb{P}_T(\mathcal{O}(N)\oplus \mathcal{O})$, $H_{\mid H_0}$ is a tautological divisor, and $S_1$ is linearly equivalent to $2H_{\mid H_0}$. Moreover, the surfaces $\{S_1, S_2, S_3\}$ are pairwise disjoint and fiber-wise in general position.

Let $h\colon X\to Z$ be the blow-up along $\{S_1, S_2, S_3\}$, and set $\sigma:=h \circ \varphi\colon X\to T$. We denote by $E_i$ the exceptional divisors over $S_i$, $i=1,2,3$, and by $\widetilde H_0$ and $\widetilde D$ the strict transforms of $H_0$ and $D$ in $X$. 
 
We now recall the description of the relative cone $\NE(\sigma)$ and its elementary contractions, which are all divisorial. The corresponding exceptional divisors will be our key to study the $K$-polystability of the varieties obtained via Construction B. We refer to \cite[$\S$6.3]{CRS} for details. 

Let $t\in T\setminus \Delta$, so that $X_t:=\sigma^{-1}(t)$ is a smooth del Pezzo surface of degree $5$ and a smooth $\sigma$-fiber. Denote by $\{p_1,p_1', p_2, p_3\} \in Z_{t}:=\varphi^{-1}(t)$ the  points blown-up by $h_{|X_{t}}\colon X_{t}\to Z_{t}$, where $p_i= S_i \cap Z_{t}$ for $i=2,3$, and $\{p_1, p_1'\}=S_1 \cap Z_{t}$. The $5$-dimensional cone  $\NE(X_{t})$ is generated by the classes of the ten $(-1)$-curves in $X_{t}$, given by the exceptional curves and the transforms of the lines through two blown-up points. 
We denote by
$e_i$ (respectively $e_1'$) the exceptional curve over $p_i$ (respectively $p_1'$), and $\ell_{i,j}$
 (respectively $\ell_{1,1'}$, $\ell_{1',i}$ for $i=2,3$)
the transform of the line $\overline{p_ip_j}$
(respectively $\overline{p_1p_1'}$, $\overline{p_1'p_i}$ for $i=2,3$). Let $\iota\colon X_{t}\hookrightarrow 
X$ be the inclusion; by \cite[Lemma 6.4]{CRS} one has that every relative elementary contraction of $X/T$ restricts to a non-trivial contraction of $X_{t}$, and
$\iota_*\NE(X_{t})=\NE(\sigma)$.

Figure \ref{figuraconoB} shows the $3$-dimensional polytope obtained as a hyperplane section of the $4$-dimensional cone $\NE(\sigma)$, which has $7$ extremal rays, and their generators.
By \cite[Thm.~1.2]{Wisn91I} we deduce that every relative elementary contraction of $\NE(\sigma)$ is the blow-up of a smooth variety along a smooth codimension 2 subvariety. The contraction corresponding to  $[e_1]=[e_1']$ (resp.\ $[e_2], [e_3]$) is the blow-down of $E_1$ (resp.\ $E_2, E_3$), while the contractions corresponding to $[\ell_{1,1'}]$ and $[\ell_{2,3}]$ have respectively exceptional divisors  $\widetilde H_0$ and $\widetilde D$.
Moreover, we denote by $G_i$ the exceptional divisor of the contraction corresponding to $[\ell_{1,i}]=[\ell_{1',i}]$ for $i=2,3$; by construction, $G_i$ has a $\pr^1$-bundle structure over $S_1$ whose fibers are numerically equivalent to $\ell_{1,i}$ and $\ell_{1',i}$ for $i=2,3$.

Lastly, we observe that $E_1 \cong G_2 \cong G_3$ and that $E_2\cong E_3\cong \w H_0$.
\begin{figure} \label{fig:figure1}
\includegraphics[width=0.4\columnwidth]{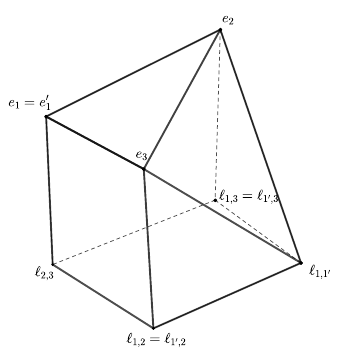}
\caption{A section of $\NE(\sigma)$}\label{figuraconoB}
\end{figure}
\subsection{Construction B: relations among exceptional divisors} \label{sec:symmetry} In this section we refer to \cite[$\S$3.8.2]{TesiSecci}. 

\begin{remark}\label{rem:symmetry}
By \cite[Proposition 6.6]{CRS}, we know that $\sigma \colon X \to T$ has three factorizations of the form $X \xrightarrow{h} Z \xrightarrow{\varphi} T$, where $h \colon X \to Z$ is the divisorial contraction of $\{E_1, E_2, E_3\}$, $\{G_2, E_3, \w H_0\}$ or $\{G_3, E_2, \w H_0\}$, and $Z\xrightarrow{\varphi}T$ is isomorphic to the $\pr^2$-bundle from $\S$\ref{sec:relB}. In fact, there is a $\mathbb Z_3$-action on the set of $\sigma$-exceptional divisors
$$\{E_1, E_2, E_3, \w H_0, \w D, G_2, G_3\}$$
induced by an automorphism of a general $\sigma$-fiber $X_t$ (see \cite[$\S$8.5.4]{Dolg12} for the description of $\Aut(X_t)$), which in turn it extends to an automorphism of $X$ over $T$. This action corresponds to the permutation $(1,2,3)$ on the triplets $(E_1, G_2, G_3)$ and $(E_2, E_3, \w H_0)$, while $\w D$ remains invariant.

The symmetry on the $\sigma$-exceptional divisors given by the three factorizations of $\sigma$ allows us, for instance, to deduce computations on $E_3$ and $\w H_0$ from computations on $E_2$. This will be a key tool for the computation in $\S$\ref{sec:non_toric}. Moreover, the unique behaviour of $\w D$ among all $\sigma$-exceptional divisors led us to the computation of its $\beta$-invariant.
\end{remark}

\medskip

Recall that $H_0 -\varphi^*N\sim D$, $S_1 \subset H_0$ and $S_2, S_3 \subset D$, so that the pull-back $h^*$ and the above Remark yield the following relations among the $\sigma$-exceptional divisors:
\begin{gather*}
\w H_0 +E_1 -\sigma^*N \sim \w D +E_2+E_3, \qquad
E_2 +G_2 -\sigma^*N \sim \w D +\w H_0+E_3, \\
E_3 +G_3 -\sigma^*N \sim \w D +\w H_0+E_2.
\end{gather*}
Moreover, $E_2$, $E_3$, $\w H_0$ and $\w D$ are $\pr^1$-bundles over $T$, while $E_1$, $G_2$ and $G_3$ are $\pr^1$-bundles over $S_1$. We have that:
\begin{itemize}
\item[(i)] $\w H_0 \cong \pr_T(-K_T \oplus -K_T-N)$ and ${-K_X}_{|\w H_0}$ is the tautological divisor; the same holds for $E_2$ and $E_3$.
\item[(ii)] $\w D\cong \pr_T(-K_T-N \oplus -K_T-N)$ and ${-K_X}_{|\w D}$ is the tautological divisor.
\end{itemize}
Note that $E_2$, $E_3$ and $\w H_0$ are pairwise disjoint, and that their intersection with $\w D$ is a section $\{pt\}\times T$ of $\w D$. As a divisor in $\w D$, this intersections correspond to surjections $\ol(-K_T-N) \oplus \ol(-K_T-N) \to \ol(-K_T-N)$, while they correspond to the projection $\ol(-K_T) \oplus \ol(-K_T-N) \to \ol(-K_T-N)$, as a divisor in $E_2$, $E_3$ and $\w H_0$.

\medskip

Finally, we can write $-K_X$ as
\begin{equation}\label{equation:anticanonical}
-K_X \sim \sigma^*(-K_T+N)+\w{H}_0+2\widetilde D+E_2+E_3.
\end{equation}

\section{Proof of Theorem \ref{thm:main}} \label{sec:proof}
In this section we show Theorem \ref{thm:main}. We keep the notation introduced in the previous section. The case $\delta\geq 4$ has been explained in Remark \ref{rem:delta4}, thus from now on we consider the case $\delta=3$. 

\subsection{Toric case} \label{sec:toric_proof}
We recall from Theorem \ref{classification} that there are 14 families of toric Fano 4-folds with $\delta=3$, and from Remark \ref{rem:toric} that all of them arise via Construction A. The aim of this section is to deduce which ones among them are K-polystable. Our conclusion will be the following:
\begin{proposition} \label{prop:toric_case} Let $X$ be a toric Fano 4-fold with $\delta_X=3$. Then it is $K$-polystable if and only if it is one of the following varieties:
\begin{itemize}\renewcommand\labelitemi{\tiny \textbullet}
\item $X\cong \mathbb{P}^2\times F$, where $F$ is the blow-up of $\mathbb{P}^2$ along three non-collinear points;
\item $X\cong \mathbb{P}^1\times\mathbb{P}^1 \times F$;
\item $X$ is the blow-up of $\mathbb{P}^1\times \mathbb{P}_{\mathbb{P}^1\times\mathbb{P}^1}(\mathcal{O}\oplus \mathcal{O}(1,-1))$ along two surfaces isomorphic to $\mathbb{P}^1\times \mathbb{P}^1$.
\item $X\cong F\times F$.
\end{itemize}
\end{proposition}

In order to prove the above result, we recall that Gorenstein toric Fano varieties correspond to reflexive lattice polytopes, that is those for which the dual 
is also a lattice polytope. We will make use of the following characterization of K-polystability for toric Fano varieties. 
 
\begin{lemma} \label{lem:toric} \cite[Corollary 1.2]{Berman2016}
Let $X_P$ be a toric Fano variety associated to a reflexive polytope $P$. Then, $X_P$ is $K$-polystable if
and only if $0$ is the barycenter of $P$.
\end{lemma}
In the following proof we follow Batyrev's notation  \cite{Baty} for the toric Fano 4-folds of Theorem \ref{classification} with $\rho=5,6$: type $K$ for varieties of Theorem \ref{classification} having $\rho=5$, and type $U$ for the ones with $\rho=6$. 
\begin{proof}[Proof of Proposition \ref{prop:toric_case}.] Assume that $X$ is a product of surfaces. If $\rho_X=5$, then $X=K_4\cong \mathbb{P}^2\times F$ is K-polystable by Corollary \ref{cor_delPezzo} and Lemma \ref{lem:product}. If $\rho_X=6$, then either $X=U_4\cong \mathbb{F}_1\times F$ or $X=U_5\cong \mathbb{P}^1\times\mathbb{P}^1 \times F$, and applying the same results we deduce that among them only $U_5$ is K-polystable. For the same reason, and by Theorem \ref{classification}, we deduce that $X$ is not K-polystable if $\rho_X=7$, while it is K-polystable if $\rho_X=8$, namely if $X\cong F\times F$.	

Assume now that $X$ is not a product of surfaces. In view of Lemma \ref{lem:toric} we are left to check whether $0$ corresponds to the barycenter of the polytopes corresponding to the remaining varieties of our classification. To this end, we use the Graded ring database (see \cite{GRD}), giving the invariants of these varieties (computed in \cite{Baty}) as inputs. It turns out that among them, the only K-polystable variety is $U_8$, that is the blow-up of $\mathbb{P}^1\times \mathbb{P}_{\mathbb{P}^1\times\mathbb{P}^1}(\mathcal{O}\oplus \mathcal{O}(1,-1))$ along two surfaces isomorphic to $\mathbb{P}^1\times \mathbb{P}^1$.
\end{proof}

\subsection{Non-toric case} \label{sec:non_toric}
We now prove prove that among the five possible families of non-toric Fano 4-folds having $\delta=3$ (see Theorem \ref{classification} and Remark \ref{rem:toric}), only one is $K$-polystable. More precisely, after our discussion we will deduce the following:

\begin{proposition} \label{prop:non-toric} Let $X$ be a non-toric Fano 4-fold with $\delta_X=3$. Then $X$ is $K$-polystable if and only if $X\cong \mathbb{P}^1\times Y$, where $Y$ is the blow-up of $\mathbb{P}^3$ along a disjoint union of a line and a conic, and along two non-trivial fibers of the exceptional divisor over the blown-up line. Otherwise, $X$ is $K$-unstable.
\end{proposition}
Note that $Y$ above is $\ma N^{\circ} 5.1$ in the notation of \cite{Araujo23}.
We recall that all non-toric Fano 4-folds of Theorem \ref{classification} arise from Construction B, see Remark \ref{rem:toric}. In particular, $X$ from Proposition \ref{prop:non-toric} is obtained via this construction taking $T\cong \mathbb{P}^1\times \mathbb{P}^1$ and $N\in |\mathcal{O}_{\mathbb{P}^1\times \mathbb{P}^1}(0,1)|$. 
In order to prove Proposition \ref{prop:non-toric}, the first objective is to compute the $\beta$-invariant of $\widetilde{D}$ (see $\S$\ref{sec:relB}), and to this end we will show the following result. 

\begin{proposition} \label{prop:beta}
Let $X$ be a non-toric Fano 4-fold with $\delta_X=3$. Set $a=-K_X^4$; $b=N^2$; $c=(-K_T-N)^2$; $d=N \cdot (-K_T-N)$; $e=(-K_T+N)^2$; $f=N\cdot (-K_T+N)$. Then: $$\beta(\widetilde{D})=\frac{1}{a}\bigl(\frac{2}{5}b+8c+6d-4e+4f\bigr).$$
\end{proposition}

We start with some preliminary computations that follow from $\S$\ref{sec:relB} and \cite[Appendix A]{Hart}; we will use these to prove the lemmas below. 

\begin{remark}\label{rem:computations}
Recall from Remark \ref{rem:symmetry} that there is a symmetry among the exceptional divisors $\{E_2, E_3, \w H_0\}$. Denote by $\eta$ a tautological divisor of $\mathbb{P}_T(N\oplus N)$ and by $\xi$ a tautological divisor of $\mathbb{P}_T(N\oplus \ol)$. Then, 
\begin{itemize}\renewcommand\labelitemi{\tiny \textbullet}
\item $\widetilde{D}_{\mid \widetilde{D}}=-\eta$, so that ${\w D}^3\sim \eta\cdot\sigma_{|\w D}^*(2N)- \sigma_{|\w D}^*(N)^2$ and ${\w D}^4=-3N^2$;
\item ${\tilde{H_0}}_{|\w H_0}=-\xi$, so that $(\w H_0)^3\sim \xi\cdot\sigma_{|E_2}^*(N)$ and $(\w H_0)^4=-N^2$; the same holds for $E_2$ and $E_3$;
\item ${-K_X}_{|\w D} \sim \eta + \sigma_{|\w D}^*(-K_T-2N)$;
\item $({-K_X}_{|\w D})^2\sim{-K_X}_{|\w D}\cdot \sigma_{|\w D}^*(-2K_T-2N) - \sigma_{|\w D}^*(-K_T-N)^2$;
\item ${-K_X}_{|\w H_0}\sim \xi+\sigma_{|\widetilde H_0}^*(-K_T-N)$; the same holds for $E_2$ and $E_3$;
\item $(\w H_0)^2\cdot \w D\sim 0$ and $(\w H_0)^3\cdot \w D=0$; the same holds for $E_2$ and $E_3$;
\item $(\sigma^*M)^i\sim0$ for all $M\in \Pic(T)$ and $i=3,4$.
\end{itemize}
Note that in all the examples of Fano 4-folds with $\delta=3$ obtained via Construction B, the divisor $N$ is nef (see \cite[$\S 7$]{CRS}), therefore $\eta$ and $\xi$ are nef as well.
\end{remark}

We first describe the Zariski decomposition of the divisor $-K_X-t\widetilde{D}$ when $t\geq 0$.
\begin{proposition}\label{prop:ZD}
Let $X$ be a non-toric  Fano 4-fold with $\delta_X=3$. The Zariski decomposition of the divisor $-K_X-t\widetilde{D}$ is given by
\begin{center}
$$
P(t)=
\begin{cases}
-K_X-t\widetilde{D}, & t\in \left[0,1\right]\\
H(t), & t\in \left(1,2\right]
\end{cases}
$$
\end{center}
where $H(t)=[\sigma^*(-K_T+N)+(2-t)\widetilde{D}]+(2-t)(\widetilde{H}_0+E_2+E_3)$, and the pseudoeffective threshold of $\w D$ with respect to $-K_X$ is $\tau(\w D)=2$. 
\end{proposition}

The proof of Proposition \ref{prop:ZD} is a consequence of the lemmas below and $\S$\ref{sec:ZD}.

\begin{lemma} \label{lemma:step1} The restriction of $-K_X-t\widetilde{D}$ to $\w H_0$, $E_2$ and $E_3$ is nef for $0\leq t\leq 1$, while $(-K_X-t\widetilde{D})_{\mid \widetilde{D}}$ is nef for $t\geq 0$.
\end{lemma}
\begin{proof}
Recall that by construction, $-K_T\pm N$ is an ample divisor on $T$, so $-K_T+ sN$ is ample for $-1 \leq s \leq 1$. Thus, $(-K_X-t\widetilde{D})_{\mid \widetilde{H}_0}\sim (1-t)\xi+\sigma_{|\widetilde H_0}^*(-K_T+(t-1)N)$ is nef for $t\leq 1$.
Similarly, $(-K_X-t\widetilde{D})_{\mid \widetilde{D}}\sim(1+t)(\eta-\sigma_{\mid \w D}^*N)+\sigma_{\mid \w D}^*(-K_T+(t-1)N)$; the claim follows since $\eta-\sigma_{\mid \w D}^*N$ is the tautological divisor of $\mathbb{P}_T(\mathcal{O}\oplus \mathcal{O})$, $-K_T-N$ is ample, and $N$ is nef.
\end{proof}
\begin{remark} \label{rmk:strictly_pos} Let $\Gamma$ be an irreducible curve not contained in $\widetilde{H}_0\cup E_2 \cup E_3 \cup \widetilde{D}$. If $\Gamma$ is contracted by $\sigma$, then by construction $\w H_0\cdot\Gamma$, $E_2\cdot\Gamma$, $E_3\cdot\Gamma$, $\w D\cdot\Gamma\geq0$ and at least one inequality is strict. 
\end{remark}
\begin{lemma} \label{lem:nef} The divisor $-K_X-t\widetilde{D}$ is nef for $0\leq t\leq 1$.
\end{lemma}

\begin{proof} Assume $t>1$. Let $\ell$ be a fiber of the restriction $\sigma_{\mid \widetilde{H}_0}\colon \widetilde{H}_0\to T$. Since $\ell$ is a fiber of the exceptional divisor of a smoth blow-up (see $\S$2.4), one has $(-K_X-t\widetilde{D})\cdot \ell=1-t<0$. Thus, we may assume $t\leq 1$. If $\Gamma$ is an irreducible curve contained in $\widetilde{H}_0\cup E_2 \cup E_3 \cup \widetilde{D}$, then $(-K_X-t\widetilde{D})\cdot \Gamma \geq 0$ by Lemma \ref{lemma:step1}. Otherwise, using equation \eqref{equation:anticanonical} in $\S$\ref{sec:symmetry}, one has $$(-K_X-t\widetilde{D})\cdot \Gamma=[\sigma^*(-K_T+N)+(2-t)\widetilde{D}+(\widetilde{H}_0+E_2+E_3)]\cdot \Gamma >0,$$ and this follows from Remark \ref{rmk:strictly_pos} and the ampleness of $-K_T+N$. 
\end{proof}
\begin{lemma} \label{lem:supp_div} The divisor $-K_X-\widetilde{D}$ is a supporting divisor of the birational contraction $X\to W$ associated to the facet $\langle [e_2],[e_3],[\ell_{1,1'}] \rangle$ of $\NE{(\sigma)}$. 
\end{lemma}
\begin{proof} By \cite[Remark 6.7]{CRS} we know that the contraction $X\to W$ is divisorial, and we show that $-K_X-\widetilde{D}$ is a supporting divisor. From Lemma \ref{lem:nef} and its proof, one has that $-K_X-\widetilde{D}$ is nef and not ample, and the curves on which it vanishes are contained in $\widetilde{H}_0\cup E_2 \cup E_3 \cup \widetilde{D}$. Furthermore, we see from the proof of Lemma \ref{lemma:step1} that $-K_X-\widetilde{D}$ vanishes exactly on the curves contained in the intersection of $\widetilde{H}_0\cup E_2 \cup E_3$ with a general $\sigma$-fiber. This gives the claim.
\end{proof}

\begin{proof}[Proof of Proposition \ref{prop:ZD}]  In view of Lemma \ref{lem:nef}, we are left to understand the decomposition of $-K_X-t\widetilde{D}$ into positive and negative part for $t\geq1$.
By Lemma \ref{lem:supp_div}, being $\widetilde{H}_0\cup E_2\cup E_3$ the exceptional locus of the divisorial contraction having $-K_X-\widetilde{D}$ as a supporting divisor, we need to determine $a(t)$, $b(t)$, $c(t)\geq 0$ and all values of $t\geq 1$ such that the positive part of the Zariski decomposition of $-K_X-t\widetilde{D}$ is
$$P(t)=-K_X-t\widetilde{D}-a(t)\widetilde{H}_0-b(t) E_2-c(t) E_3.$$
Denote by $h_0$, $e_2$, $e_3$ the intersection of a general $\sigma$-fiber with $\widetilde{H}_0$, $E_2$, and $E_3$ respectively. Requiring that $P(t)$ has zero intersection with $h_0$, $e_2$ and $e_3$ yields $a(t)=b(t)=c(t)=t-1$. Set $H(t):= -K_X-t\widetilde{D}-(t-1)(\widetilde{H}_0+ E_2+ E_3)$.
By equation \eqref{equation:anticanonical} in $\S$\ref{sec:symmetry}, we deduce that
$$H(t)=[\sigma^*(-K_T+N)+(2-t)\widetilde{D}]+(2-t)(\widetilde{H}_0+E_2+E_3).$$
Let $\Gamma$ be the intersection of $\widetilde{D}$ with a general $\sigma$-fiber, so that $H(t)\cdot \Gamma=2(2-t)$; thus, $H(t)$ is not nef for $t>2$. Finally, we see that $H(t)$ is nef for $t\leq 2$, and this follows from Remark \ref{rmk:strictly_pos} and the ampleness of $-K_T+N$. 

Since $H(2) \sim \sigma^* (-K_T+N)$ is a nef and not big divisor, we deduce that $\tau(\w D)=2$, hence our claim. 
\end{proof}

Finally, we are able to determine $S(\w D)$ (see $\S$\ref{sub_K_stab} for its definition). We split the computation into the following lemmas.
\begin{lemma}\label{lem:firstpart}
Notation as in Proposition \ref{prop:beta}. Then,
$$\int_{0}^{1} (-K_X-t\w D)^4dt = a- \frac{8}{5}b-8c-6d.$$
\end{lemma}
\begin{proof}
We compute $(-K_X-t\w D)^4$. By Remark \ref{rem:computations}, we have:
\begin{itemize}\renewcommand\labelitemi{\tiny \textbullet}
\item $-K_X^3\cdot \w D=({-K_X}_{|\w D})^3=3(-K_T-N)^2$;
\item $-K_X^2\cdot \w D^2=({-K_X}_{|\w D})^2\cdot(-\eta)=-(-K_T-N)^2-2N\cdot(-K_T-N)$;
\item $-K_X\cdot \w D^3=({-K_X}_{|\w D})\cdot \eta^2=2N\cdot(-K_T-N)+N^2$.
\end{itemize}
Therefore,
$$(-K_X-t\w D)^4=a-12ct-6(c+2d)t^2-4(b+2d)t^3-3bt^4$$
and the claim follows.
\end{proof}

\begin{lemma}\label{lem:secondpart}
Notation as in Proposition \ref{prop:beta}. Then,
$$\int_{1}^{2} H(t)^4 dt = \frac{6}{5}b-4f+4e.$$
\end{lemma}
\begin{proof}
We recall that $H(t)=[\sigma^*(-K_T+N)+(2-t)\widetilde{D}]+(2-t)(\widetilde{H}_0+E_2+E_3)$. In order to obtain $H(t)^4$, we compute the intersections
$$(2-t)^i[\sigma^*(-K_T+N)+(2-t)\widetilde{D}]^{4-i}\cdot(\widetilde{H}_0+E_2+E_3)^i,$$
for $i=0,\dots,4$. We use Remark \ref{rem:computations} for the following computations.

\begin{step}
$[\sigma^*(-K_T+N)+(2-t)\widetilde{D}]^4=-6e(2-t)^2+8f(2-t)^3-3b(2-t)^4$.
\end{step}
Indeed:
\begin{itemize}\renewcommand\labelitemi{\tiny \textbullet}
\item $\sigma^*(-K_T+N)^2\cdot \w D^2=\sigma_{|\w D}^*(-K_T+N)^2\cdot(-\eta)=-(-K_T-N)^2$;
\item $\sigma^*(-K_T+N)\cdot \w D^3=\sigma_{|\w D}^*(-K_T+N)\cdot\eta^2= 2N\cdot(-K_T+N)$.
\end{itemize}

\begin{step}
$[\sigma^*(-K_T+N)+(2-t)\widetilde{D}]^3\cdot(\w H_0+E_2+E_3)=9e(2-t)-9f(2-t)^2+3b(2-t)^3$.
\end{step}
Recall that ${\w H_0}\cap{\w D}$ is a section $\{pt\}\times T$ of $\w D$. By restricting to $\w D$ we obtain:
\begin{itemize}\renewcommand\labelitemi{\tiny \textbullet}
\item $\sigma^*(-K_T+N)^2\cdot \widetilde{D} \cdot(\w H_0+E_2+E_3)=3(-K_T+N^2)$;
\item $\sigma^*(-K_T+N)\cdot \widetilde{D}^2 \cdot(\w H_0+E_2+E_3)=-3N\cdot(-K_T+N)$;
\item $\widetilde{D}^3 \cdot(\w H_0+E_2+E_3) \cdot 3N^2$.
\end{itemize}

\begin{step}
$[\sigma^*(-K_T+N)+(2-t)\widetilde{D}]^2\cdot(\w H_0+E_2+E_3)^2=-3e$.
\end{step}
Recall that $E_2$, $E_3$ and $\w H_0$ are pairwise disjoint. Thus:
\begin{itemize}\renewcommand\labelitemi{\tiny \textbullet}
\item $\sigma^*(-K_T+N)^2 \cdot[(\w H_0)^2+(E_2)^2+(E_3)^2]=-3(-K_T+N)^2$;
\item $\sigma^*(-K_T+N)\cdot \w D \cdot[(\w H_0)^2+(E_2)^2+(E_3)^2]=0$;
\item $\w D^2 \cdot[(\w H_0)^2+(E_2)^2+(E_3)^2]=0$.
\end{itemize}

\begin{step}
$[\sigma^*(-K_T+N)+(2-t)\widetilde{D}]\cdot(\w H_0+E_2+E_3)^3=3f$.
\end{step}
Indeed:
\begin{itemize}\renewcommand\labelitemi{\tiny \textbullet}
\item $\sigma^*(-K_T+N)\cdot[(\w H_0)^3+(E_2)^3+(E_3)^3]=3N\cdot (-K_T+N)$;
\item $\w D \cdot[(\w H_0)^3+(E_2)^3+(E_3)^3]=0$.
\end{itemize}
\vspace{0.3cm}
We conclude that
$$H(t)^4=6b(2-t)^4-16f(2-t)^3+12e(2-t)^2,$$
and the claim follows.
\end{proof}

\begin{proof}[Proof of Proposition \ref{prop:beta}.]
We compute $\beta(\w D)=A(\w D)-S(\w D)$. Since $\w D\subset X$ is a prime divisor on $X$, we have $A(\w D)=1$. Moreover, due to Proposition \ref{prop:ZD}, we can compute
$$a\cdot S(\w D)= \int_{0}^{2} \vol(-K_X-t\w D) dt$$
by splitting it as
$$\int_{0}^{1} (-K_X-t\w D)^4 dt \ + \int_{1}^{2} H(t)^4dt.$$
Thus, the claim follows from Lemma \ref{lem:firstpart} and Lemma \ref{lem:secondpart}.
\end{proof}

We now apply Proposition \ref{prop:beta} and conclude this section with the proof of Proposition \ref{prop:non-toric}.

\begin{proof}[Proof of Proposition \ref{prop:non-toric}] Assume that $X$ is a product. Then, by the classification of Fano $4$-folds having $\delta=3$ (see Theorem \ref{classification} and Remark \ref{rem:toric}) it follows that $X\cong \mathbb{P}^1\times Y$ with $Y$ being as in the statement. Recall that $Y$ is $\ma N^{\circ} 5.1$ in the notation of \cite{Araujo23}. By \cite[Main Theorem]{Araujo23} we know that $Y$ is K-polystable, then using Lemma \ref{lem:product} we conclude that $X$ is K-polystable.

Suppose now that $X$ is not a product. We show that for all the remaining four families of varieties of our classification we have $\beta(\widetilde{D})<0$, so that we conclude by Theorem \ref{thm_beta} that they are not K-semistable, hence not K-polystable. Being $a>0$, in view of Proposition \ref{prop:beta} we are left to prove that $\lambda:=\frac{2}{5}b+8c+6d-4e+4f<0$. 

Assume first that $\rho_X=5$. By construction B, one has $T=\mathbb{P}^2$ and by the proof of \cite[Theorem 1.3]{CR} we know that either $N=\mathcal{O}(1)$ or $N=\mathcal{O}(2)$. Using the numerical invariants of the corresponding varieties computed in \cite[Table 3.4]{CR}, we see that $\lambda=-\frac{18}{5}$ in the first case and $\lambda=-\frac{192}{5}$ in the second case. 

Assume now that $\rho_X=6$. Construction B gives either $T=\mathbb{F}_1$ or $T=\mathbb{P}^1\times \mathbb{P}^1$. In the first case, by the proof of \cite[Proposition 7.1]{CRS} we know that $N=\pi^*L$, where $\pi\colon \mathbb{F}_1\to \mathbb{P}^2$ is the blow-up and $L$ general line in $\mathbb{P}^2$. For this variety, using the numerical invariants of \cite[Table 7.1]{CRS} we get $\lambda=-\frac{38}{5}$. Otherwise, by the proof of the same proposition we have $N=\mathcal{O}(1,1)$, and we obtain that $\lambda=-\frac{96}{5}$, hence our claim. 
\end{proof}

\subsection{Proof of Theorem \ref{thm:main}} We obtain the proof of Theorem \ref{thm:main} as a direct consequence of Remark \ref{rem:delta4}, of the classification theorem of Fano $4$-folds having $\delta=3$ (see Theorem \ref{classification}, Remark \ref{rem:toric}) and Propositions \ref{prop:toric_case}, \ref{prop:non-toric}. We summarize our results in Section \ref{sec:tables}: Table \ref{table:delta3} gathers all Fano 4-folds with $\delta=3$, while Fano 4-folds with $\delta \geq 4$ appear in Table \ref{table:high-delta}. \qed

\section{Casagrande-Druel varieties}  \label{sec:delta2}
In this section we recap the construction and the classification in dimension $4$ of the so-called Casagrande-Druel Fano varieties. They have $\delta\geq2$. This construction was first introduced in \cite[Example 3.4, Theorem 3.8]{CD} to give a geometrical characterization of all Fano manifolds with $\rho=3$ having a prime divisor of Picard number $1$, i.e.\ Fano manifolds with $\rho=3$ and $\delta=2$, and the $4$-dimensional case was classified in \cite{Secci23}. The construction from \cite{CD} was first generalized in \cite{CD-stability}, and then further generalized in \cite[Construction A, $\S2$]{Pier25}, where the author completed the classification of smooth Casagrande-Druel Fano $4$-folds with $\delta=2$.

\subsection{Casagrande-Druel Fano $4$-folds: construction and classification.}\label{CD-construction}
We follow the notation from \cite[$\S2$]{Pier25}. Let $Z$ be a smooth Fano $3$-fold; fix $D\in\Pic(Z)$ and a smooth irreducible hypersurface $A\subset Z$, and let $Y:=\pr_Z(\ol_Z\oplus\ol_Z(D))$ with projection $\pi\colon Y \to Z$.
Let $G_Y$ and $\hat G_Y$ be sections of $\pi$ corresponding to the projections $\ol\oplus\ol(D)\twoheadrightarrow \ol$ and $\ol\oplus\ol(D)\twoheadrightarrow \ol(D)$, respectively, so that $G_Y\cap \hat G_Y=\emptyset$. Let $\sigma\colon X\to Y$ be the blow-up along $\hat G_Y \cap \pi^{-1}(A)$, and let $E$ be the exceptional divisor.
Then $\phi:=\pi\circ\sigma\colon X\to Z$ is a conic bundle with discriminant divisor $\Delta_\phi=A$. Denote by $G$ and $\hat G$ the strict transforms in $X$ of $G_Y$ and $\hat G_Y$, respectively. It follows that
\begin{equation}\label{eq:normal G}
\ma N_{G/X}\cong\ol_Z(-D)\, ,\quad \ma N_{\hat G/X}\cong\ol_Z(D-A)\quad \text{and}\quad G\cap\hat G=\emptyset\, .
\end{equation}

There exists a second factorization of $\phi$ as $\hat\pi\circ\hat\sigma$, where $\hat Y:=\pr_Z(\ol_Z\oplus\ol_Z(A-D))$ with projection $\hat\pi\colon \hat Y\to Z$, and $\hat\sigma\colon X\to \hat Y$ is the blow-up along $\hat\sigma(G)\cap\hat\pi^{-1}(A)$, with exceptional divisor $\hat E$ coinciding with the strict transform in $X$ of $\pi^{-1}(A)\subset Y$.
$$
\xymatrix{
 & X \ar[rd]^{\sigma}\ar[ld]_{\hat \sigma} \ar[dd]^{\phi}& \\
\hat Y \ar[rd]_{\hat\pi} & & Y \ar[ld]^{\pi} \\
 & Z & 
}
$$

We list some properties of $X$:
\begin{itemize}\renewcommand\labelitemi{\tiny \textbullet}
\item the smooth $4$-folds constructed as above from $(Z;A,D)$ and $(Z;A,A-D)$ coincide, with all objects with and without a hat $(\ \hat{}\ )$ interchanged;
\item $X$ is Fano if and only if $-K_Z-D$ and $-K_Z-A+D$ are ample on $Z$, \cite[Prop.\ 2.3]{Pier25};
\item if $X$ is Fano, then $\delta_X \geq 2$, \cite[Lemma 2.4]{Pier25}.
\end{itemize}

Finally, we collect here useful relations among divisor classes, and intersection numbers. Let $G$, $\hat G$, $E$, $\hat E$ as above, and let $\xi= \hat G_Y$ and $\hat \xi= \hat \sigma (G)$ be the tautological classes of $Y=\pr_Z(\ol_Z\oplus\ol_Z(D))$ and $\hat Y=\pr_Z(\ol_Z\oplus\ol_Z(A-D))$ respectively. Then,
\begin{equation}\label{eq:relationsCD}
\begin{gathered}
\xi-\pi^*D\sim G_Y\, , \qquad \hat\xi-\hat\pi^*(A-D)\sim \hat\sigma (\hat G) \, , \\
\sigma^*\xi\sim \hat G +E\, , \quad \hat\sigma^*\hat\xi\sim G + \hat E \, , \quad \sigma^*G_Y \sim G \, , \quad \hat \sigma^*(\hat\sigma (\hat G)) \sim \hat G \, , \\
-K_X\sim \sigma^*(-K_Y)-E \sim\hat \sigma^*(-K_{\hat Y})-\hat E\sim \phi^*(-K_Z) +G+\hat G\, .
\end{gathered}
\end{equation}
Let $e\subset E$ and $\hat e\subset \hat E$ be fibers of the blow-ups $\sigma$ and $\hat\sigma$, respectively. Then we obtain the following intersection numbers: 

\begin{table}[!htbp]
\begin{tabular}{|c|c|c|c|c|}
\hline
\rule{0pt}{2.7ex} $\cdot$ & $E$ & $\hat E$ & $G$ & $\hat G$ \\
\hline
\rule{0pt}{2.7ex} $e$ & -1 & 1 & 0 & 1 \\
\hline
\rule{0pt}{2.7ex} $\hat e$ & 1 & -1 & 1 & 0 \\
\hline
\end{tabular}
\end{table}

\begin{remark}
In \cite{CD} and \cite{Secci23}, where $Z$ has Picard number $1$ and $\Pic(Z)=\Z\cdot\ol_Z(1)$ with $\ol_Z(1)$ the ample generator, the role of $A$ and $D$ are somewhat reversed compared to \cite{Pier25}: indeed, $A$ is a smooth irreducible hypersurface in $|\ol_Z(d)|$ and $D\sim \ol_Z(a)$, so that $Y:=\pr_Z(\ol_Z\oplus\ol_Z(a))$. Thus, the pair of divisors $(A,D)$ from \cite{Pier25} corresponds to the pair of integers $(d,a)$ in \cite{CD} and \cite{Secci23}.
\end{remark}

We finally recall the classification results of smooth Casagrande-Druel Fano $4$-folds with $\delta=2$. Note that $\rho\geq 3$ by construction and $\rho\leq 6$ by \cite[Theorem 8.2]{Pier25}; more generally, the bound $\rho\leq6$ also follows from \cite[Theorem 1.8]{CS24}.
\begin{thm}[\cite{Secci23}, Thm.\ 1.1]\label{thm:Saverio}
There are $28$ deformation families of smooth Casagrande-Druel Fano $4$-folds with $\delta=2$ and $\rho=3$. Among them, $3$ families are toric.
\end{thm}

\begin{thm}[\cite{Pier25}, Thm.\ 1.4]\label{thm:Pier}
There are $147$ deformation families of smooth Casagrande-Druel Fano $4$-folds with $\delta=2$ and $\rho\in\{4,5,6\}$. Among them, $43$ families are toric.
\end{thm}

\section{Proof of Theorem \ref{thm:main2}} \label{sec:proof2}
In this section we prove Theorem \ref{thm:main2}. We keep the same notation as in the previous section.

\begin{proposition}\label{prop:beta-G}
Let $X$ be a smooth Casagrande-Druel Fano $4$-fold constructed from $(Z;A,D)$ as in $\S$\ref{CD-construction}, and let $f\colon \ma N^1(Z) \to \R$ be the continuous function
$$f(\zeta)=4K_Z^2\cdot\zeta+3K_Z\cdot \zeta^2+\frac{4}{5}\zeta^3.$$
Then, $\beta(\hat G)=-\beta(G)=\frac{1}{(-K_X)^4}\bigl(f(D)-f(A-D)\bigr)$.
\end{proposition}

\begin{corollary}\label{co:beta-G}
Let $X$ be a smooth Casagrande-Druel Fano $4$-fold with $\rho_X=3$ constructed from $(Z;A,D)$ as in $\S$\ref{CD-construction}. Then, $\beta(G)=\beta(\hat G)=0$ if and only if $A\sim 2D$.
\end{corollary}

We now proceed with the Zariski decomposition of $-K_X-t\hat{G}$ and $-K_X-t{G}$ when $t\geq 0$, and then prove Proposition \ref{prop:beta-G} and Corollary \ref{co:beta-G}.
\begin{proposition}\label{prop:ZD-G}
Let $X$ be a smooth Casagrande-Druel Fano $4$-fold constructed from $(Z;A,D)$ as in $\S$\ref{CD-construction}. Then the Zariski decomposition of the divisors $-K_X-t\hat{G}$ and $-K_X-t{G}$ is given respectively by
\begin{center}
$$
P(t)=
\begin{cases}
-K_X-t\hat{G}, & t\in \left[0,1\right]\\
\sigma^*\bigl((2-t)\xi+ \pi^*(-K_Z-D)\bigr), & t\in \left(1,2\right]
\end{cases}
$$
and
$$
P(t)=
\begin{cases}
-K_X-t{G}, & t\in \left[0,1\right]\\
\hat\sigma^*\bigl((2-t)\hat\xi+\hat \pi^*(-K_Z-A+D)\bigr), & t\in \left(1,2\right]
\end{cases}
$$
\end{center}
and the pseudoeffective threshold of $\hat G$ and $G$ with respect to $-K_X$ is $\tau(\hat G)=\tau(G)=2$. 
\end{proposition}
\begin{proof}
By the simmetry of the contruction, it suffices to show the statement for $-K_X-t\hat G$; the Zariski decomposition of $-K_X-t G$ is then obtained by swapping $D$ and $A-D$, as well as all objects with and without a hat $(\ \hat{}\ )$. Recall from $\S$\ref{CD-construction} that $-K_Z-D$ and $-K_Z-A+D$ are ample on $Z$.
\begin{step}\label{step:nef-G}
Let $t\geq 0$. Then, $-K_X-t \hat G$ is nef for $0\leq t \leq 1$ and $-K_X-\hat G$ is a supporting divisor of $\sigma \colon X \to Y$.
\end{step}
Note first that $(-K_X-t \hat G)\cdot e =1-t$, thus $-K_X-t \hat G$ is not nef for $t>1$ and we can assume $t\leq 1$. By \eqref{eq:relationsCD} we can write
$$-K_X-t \hat G \sim \phi^*(-K_Z) + G + (1-t)\hat G$$
as a sum of a nef divisor and effective divisors. Thus, to show that $-K_X-t \hat G$ is nef for $0\leq t \leq 1$ we are left to prove that its restriction to $G$ and $\hat G$ is nef for $0\leq t \leq 1$.
From \eqref{eq:normal G} we have that
\begin{itemize}\renewcommand\labelitemi{\tiny \textbullet}
\item $(-K_X-t \hat G)_{|G} \sim -K_Z -D$ is ample, and
\item $(-K_X-t \hat G)_{|\hat G} \sim -K_Z + (1-t)(D-A)=t(-K_Z)+(1-t)(-K_Z-D+A)$ is ample for $0\leq t\leq 1$, as it is the sum of ample divisors.
\end{itemize}
Lastly, it follows from \eqref{eq:relationsCD} that $-K_X-\hat G\sim\sigma^*\bigl(\xi+\pi^*(-K_Z-D)\bigr)$. Thus, $\xi+\pi^*(-K_Z-D)$ is the tautological class of $\pr_Z(\ol_Z(-K_Z-D)\oplus \ol_Z(-K_Z))$ and is ample on $Y$. This concludes the proof.

\begin{step}
The positive part of the Zariski decomposition of $-K_Z-t\hat G$ for $1\leq t\leq 2$ is $P(t)=\sigma^*\bigl((2-t)\xi+ \pi^*(-K_Z-D)\bigr)$, and $-K_Z-2\hat G$ is nef and not big.
\end{step}
Assume $t\geq 1$. By Step \ref{step:nef-G}, $-K_X-\hat G$ is a supporting divisor of $\sigma \colon X \to Y$. We therefore need to determine $u(t)\geq 0$ and all values of $t\geq 1$ such that the positive part of the Zariski decomposition of $-K_X- t\hat G$ is $P(t)=-K_X-t\hat G-u(t)E$. Since $-K_X-\hat G$ is a supporting divisor of $\sigma$ we have $(-K_X-t\hat G-u(t)E)\cdot e=0$, which gives $u(t)=t-1$.
Set $H(t):=-K_X-t\hat G- (t-1)E$. Note that $H(t)\cdot \hat e = 2-t$, thus $H(t)$ is not nef for $t>2$ and we may therefore assume $t\leq 2$.
By \eqref{eq:relationsCD} we can write
$$H(t)\sim(2-t)(\hat G+E)+\phi^*(-K_Z-D)\sim \sigma^*\bigl((2-t)\xi+ \pi^*(-K_Z-D)\bigr).$$
By Step \ref{step:nef-G} we recall that $\xi+\pi^*(-K_Z-D)$ is ample on $Y$. Thus, since $(2-t)\xi+ \pi^*(-K_Z-D)\sim (2-t)\bigl(\xi+\pi^*(-K_Z-D)\bigr)+(t-1)\pi^*(-K_Z-D)$ is the sum an ample divisor and a nef divisor for $1\leq t\leq2$, it is ample. We conclude that $H(t)$ is nef for $1\leq t\leq2$ and $H(2)$ is not big. This concludes the proof.
\end{proof}

\begin{proof}[Proof of Proposition \ref{prop:beta-G}]
We compute $\beta(\hat G)=A(\hat G)-S(\hat G)$. Since $\hat G\subset X$ is a prime divisor on $X$, we have $A(\hat G)=1$. Moreover, due to Proposition \ref{prop:ZD-G}, we can compute
$$(-K_X)^4\cdot S(\hat G)= \int_{0}^{2} \vol(-K_X-t\hat G) dt$$
by splitting the integral as
$$\int_{0}^{1} (-K_X-t\hat G)^4 dt \ + \int_{1}^{2} \bigl((2-t)\xi+ \pi^*(-K_Z-D)\bigr)^4dt.$$
The statement for $\beta(\hat G)$ follows from the two steps below, and for $\beta(G)$ by the simmetry of the contruction, i.e.\ by swapping $D$ and $A-D$, and items with and without a hat $(\ \hat{}\ )$.

\begin{step}
$\int_{0}^{1} (-K_X-t\hat G)^4 dt=(-K_X)^4-2(-K_Z)^3+f(A-D)$.
\end{step}
By \eqref{eq:normal G} and \eqref{prop:beta-G} we have ${-K_X}_{|\hat G}\sim(-K_Z+D-A)$ and ${\hat G}_{|\hat G}\sim D-A$, so that
$$\alpha_i:=(-K_X)^{4-i}\cdot\hat G^i=(-K_Z+D-A)^{4-i}\cdot(D-A)^{i-1}$$
for $i\in\{1,\dots,4\}$. Therefore, $(-K_X-t\hat G)^4=(-K_X)^4-4\alpha_1t+6\alpha_2t^2-4\alpha_3t^3+\alpha_4t^4$ and
$$\int_{0}^{1} (-K_X-t\hat G)^4 dt=(-K_X)^4-2\alpha_1+2\alpha_2-\alpha_3+\frac{1}{5}\alpha_4,$$
and the statement follows from simple computations.

\begin{step}
$\int_{1}^{2} \bigl((2-t)\xi + \pi^*(-K_Z-D)\bigr)^4dt=2(-K_Z)^3-f(D)$.
\end{step}
Recall from $\S$\ref{CD-construction} that $\xi=\hat G_Y$ and that $\ma N_{\hat G_Y/Y}\sim \ol_Z(D)$, so that
$$\gamma_i:=\pi^*(-K_Z-D)^{4-i}\cdot (\hat G_Y)^i=(-K_Z-D)^{4-i}\cdot D^{i-1}$$
for $i\in\{1,\dots,4\}$ and $\pi^*(-K_Z-D)^4=0$. Therefore, $\bigl((2-t)\xi + \pi^*(-K_Z-D)\bigr)^4=4\gamma_1(2-t)+6\gamma_2(2-t)^2+4\gamma_3(2-t)^3+\gamma_4(2-t)^4$ and
$$\int_{1}^{2} \bigl((2-t)\xi + \pi^*(-K_Z-D)\bigr)^4dt=2\gamma_1+2\gamma_2+\gamma_3+\frac{1}{5}\gamma_4,$$
and the statement follows from simple computations.
\end{proof}

\begin{proof}[Proof of Corollary \ref{co:beta-G}]
It clearly follows from Proposition \ref{prop:beta-G} that $\beta(\hat G)=\beta(G)=0$  if $A\sim 2D$. Assume now $\rho_X=3$, so that $\rho_Z=1$ and $\ma N^1(Z)\cong \R$: then $f\colon \ma N^1(Z) \to \R$ is given by $f(\zeta)=4i_Z^2\zeta-3i_Z\zeta^2+\frac{4}{5}\zeta^3$, where $i_Z$ is the index\footnote{The index $i_X$ of a smooth Fano variety $X$ is the largest integer that divides $-K_X$ in $\Pic(X)$.} of $Z$. It is easy to check that $f$ is an increasing function, thus injective, and the statement follows.
\end{proof}

\subsection{Proof of Theorem \ref{thm:main2}} The proof of Theorem \ref{thm:main2} is a consequence of Proposition \ref{prop:beta-G}, the classification of Casagrande-Druel Fano 4-folds (Theorem \ref{thm:Saverio} and Theorem \ref{thm:Pier}) and \cite[Main Theorem]{Araujo23}.

Since $\beta(\hat G)=-\beta(G)$ by Proposition \ref{prop:beta-G}, either $\beta(G)\leq 0$ or $\beta(\hat G)\leq0$.

If $\rho_X=3$, then $\beta(\hat G)=\beta(G)=0$  if and only if $A\sim 2D$ (by Corollary \ref{co:beta-G}), which holds precisely when $X$ is $X^i_{a,d}$ from \cite{Secci23} with $d=2a$: this proves item (i).

If $\rho_X \geq 4$ and $A\sim 2D$, then $\beta(\hat G)=\beta(G)=0$ and $X$ is one of the following varieties from \cite{Pier25}: $\#4$-$2$, $\#5$-$2$, $\#6$-$4$, $\#7$-$17$, $\#7$-$23$, $\#8$-$16$, $\#8$-$18$, $\#9$-$4$, $\#10$-$5$, $\#11$-$2$, $\#11$-$3$, $\#12$-$2$, $\#13$-$2$, $\#14$-$4$, $\#15$-$3$, $\#16$-$4$, $\#17$-$3$, $\#18$-$4$, $\#19$-$5$, $\#19$-$9$, $\#20$-$8$, $\#20$-$10$, $\#21$-$14$, $\#21$-$21$, $\#21$-$35$, $\#21$-$37$, $\#22$-$15$, $\#22$-$17$, $\#23$-$5$.
If $\rho_X \geq 4$ and $A\not\sim 2D$ (118 out of 147 families), we explicitly compute the $\beta$-invariant of $G$ (and $\hat G$) through Proposition \ref{prop:beta-G} and see that $\beta(\hat G)=\beta(G)=0$ if and only if $X$ is $\#7$-$10$, $\#7$-$21$, $\#19$-$7$ or $\#21$-$19$. This proves item (ii).

To conclude the proof, note that $\#7$-$10$,  $\#7$-$17$, $\#8$-$16$, $\#21$-$14$ and $\#21$-$21$ are the product of $\pr^1$ with the Fano 3-fold $\ma N^{\circ} 4.7$, $\ma N^{\circ} 4.2$, $\ma N^{\circ} 4.4$, $\ma N^{\circ} 3.19$ and $\ma N^{\circ} 3.9$ (in the notation of \cite{Araujo23}), respectively. These Fano 3-folds are K-polystable by \cite[Main Theorem]{Araujo23}, and the last statement follows from Lemma \ref{lem:product}. \qed

\begin{remark}
K-polystability in the cases where $A\sim D-A$ could be studied via \cite[Theorem 1.9]{CD-stability} and \cite[Theorem 1.1]{Mallory24}. Moreover, Theorem \ref{thm:main2} is closely related to \cite[Theorem 1.3]{Mallory24}.
\end{remark}

\section{Final tables}\label{sec:tables}
We collect here the results from Section \ref{sec:proof}, the notation in the tables is as follows. In the first column we use the description of Construction B from $\S$\ref{sec:relB} for the non-toric Fano 4-folds with $\delta=3$, while we use the notation in \cite{Baty} for the toric case when $\delta=3$ and $\rho =5,6$, explicitly showing which 4-folds are product of surfaces. The second column contains the Picard number $\rho$, while in the last column with the symbol \ding{51} (resp. \ding{55}) we mean that the 4-fold is K-polystable (resp.\ not K-polystable). Table \ref{table:delta3} contains an extra column, where we write whether $\beta(\w D)$ is positive ($+ve$) or negative ($-ve$), when applicable. Note that the $4$-fold is not K-semistable when $\beta(\w D)<0$.
Finally, recall that $F'$ (resp.\ $F$) is the blow-up of $\mathbb{P}^2$ along two (resp.\ three non-collinear) points.

\begin{table}[!htbp]\caption{K-polystability of Fano 4-folds with $\delta=3$}\label{table:delta3}
\begin{tabular}{|>{$}c<{$}|>{\centering $}m{2.2em}<{$}|>{\centering $}m{2.2em}<{$}|c||}
\hline\hline
\rule{0pt}{2.7ex} \text{$4$-fold} & \rho & \beta(\w D) & K-polystable \\
\hline\hline
\multicolumn{4}{|c||}{\rule{0pt}{2.7ex}\textbf{Non-toric}} \\
\hline
\rule{0pt}{2.7ex} T=\pr^2,\ N=\ol(1) & 5 & -ve & \ding{55} \\
\hline
\rule{0pt}{2.7ex} T=\pr^2,\ N=\ol(2) & 5 & -ve & \ding{55} \\
\hline
\rule{0pt}{2.7ex} T=\pr^1\times\pr^1,\ N=\ol(0,1) & 6 & +ve& \ding{51} \\
\hline
\rule{0pt}{2.7ex} T=\pr^1\times\pr^1,\ N=\ol(1,1) & 6 & -ve & \ding{55} \\
\hline
\rule{0pt}{2.7ex} T=\mathbb F_1,\ N=\pi^*L & 6 & -ve & \ding{55} \\
\hline
\multicolumn{4}{|c||}{\rule{0pt}{2.7ex} \textbf{Toric}} \\
\hline
\rule{0pt}{2.7ex} K_1 & 5 & \text - & \ding{55} \\
\hline
\rule{0pt}{2.7ex} K_2 & 5 & \text - & \ding{55} \\
\hline
\rule{0pt}{2.7ex} K_3 & 5 & \text - & \ding{55} \\
\hline
\rule{0pt}{2.7ex} K_4 \cong \pr^2\times F & 5 & \text - & \ding{51} \\
\hline
\rule{0pt}{2.7ex} U_1 & 6 & \text - & \ding{55} \\
\hline
\rule{0pt}{2.7ex} U_2 & 6 & \text - & \ding{55} \\
\hline
\rule{0pt}{2.7ex} U_3 & 6 & \text - & \ding{55} \\
\hline
\rule{0pt}{2.7ex} U_4\cong \mathbb F_1\times F & 6 & \text - & \ding{55} \\
\hline
\rule{0pt}{2.7ex} U_5 \cong \pr^1\times\pr^1\times F & 6 & \text - & \ding{51} \\
\hline
\rule{0pt}{2.7ex} U_6 & 6 & \text - & \ding{55} \\
\hline
\rule{0pt}{2.7ex} U_7 & 6 & \text - & \ding{55} \\
\hline
\rule{0pt}{2.7ex} U_8 & 6 & \text - & \ding{51} \\
\hline
\rule{0pt}{2.7ex} F'\times F & 7 & \text - & \ding{55} \\
\hline
\rule{0pt}{2.7ex} F\times F & 8 & \text - & \ding{51} \\
\hline\hline
\end{tabular}
\end{table}

\begin{table}[!htbp]\caption{K-polystability of Fano 4-folds with $\delta\geq4$}\label{table:high-delta}
\begin{tabular}{|>{$}c<{$}|c|c||}
\hline\hline
\rule{0pt}{2.7ex} \text{$4$-fold} & $\rho$ & K-polystable \\
\hline\hline
\rule{0pt}{2.7ex} X=S\times T & \multirow{2}*{$\delta_X +\rho_T + 1$} & \multirow{2}*{\ding{51}} \\
\rule{0pt}{2.7ex} \rho_T\leq\delta_X+1,\ T\not\cong\mathbb F_1, F' & & \\
\hline
\rule{0pt}{2.7ex} X=S\times \mathbb F_1 & $\delta_X + 3$ & \ding{55} \\
\hline
\rule{0pt}{2.7ex} X=S\times F' & $\delta_X + 4$ & \ding{55} \\
\hline\hline
\end{tabular}
\end{table}

\newcommand{\etalchar}[1]{$^{#1}$}
\providecommand{\noop}[1]{}
\providecommand{\bysame}{\leavevmode\hbox to3em{\hrulefill}\thinspace}
\providecommand{\MR}{\relax\ifhmode\unskip\space\fi MR }
\providecommand{\MRhref}[2]{%
  \href{http://www.ams.org/mathscinet-getitem?mr=#1}{#2}
}
\providecommand{\href}[2]{#2}

\end{document}